\documentclass{aims}
\usepackage{amsmath}
  \usepackage{paralist}
  \usepackage{graphics} %% add this and next lines if pictures should be in esp format
  \usepackage{epsfig} %For pictures: screened artwork should be set up with an 85 or 100 line screen
\usepackage{graphicx}  \usepackage{epstopdf}%This is to transfer .eps figure to .pdf figure; please compile your paper using PDFLeTex or PDFTeXify.
 \usepackage[colorlinks=true]{hyperref}
\hypersetup{urlcolor=blue, citecolor=red}

\def\currentvolume{X}
 \def\currentissue{X}
  \def\currentyear{200X}
   \def\currentmonth{XX}
    \def\ppages{X--XX}
     \def\DOI{10.3934/xx.xx.xx.xx}

\newtheorem{theorem}{Theorem}[section]

\newtheorem{proposition}{Proposition}

\theoremstyle{definition}
\newtheorem{definition}[theorem]{Definition}
\newtheorem{remark}{Remark}

\title[Non--Uniqueness of the QHD System] %Use the shortened version of the full title
      {Non--Uniqueness of Weak Solutions of the Quantum--Hydrodynamic System}

\author[Peter Markowich and Jes\'us Sierra]{}


 \email{peter.markowich@kaust.edu.sa}
 \email{jesus.sierra@kaust.edu.sa}

\thanks{$^*$ Corresponding author: Peter Markowich}

\begin{document}
\maketitle

% Enter the first author's name and address:
\centerline{\scshape Peter Markowich$^*$ and Jes\'us Sierra}
\medskip
{\footnotesize
% please put the address of the first author
 \centerline{King Abdullah University of Science and Technology,}
   \centerline{ Box 4700, Thuwal 23955-6900, Saudi Arabia}
} % Do not forget to end the {\footnotesize by the sign }

\bigskip

% The name of the associate editor will be entered by an editorial staff
% "Communicated by the associate editor name" is not needed for special issue.
 \centerline{(Communicated by the associate editor name)}

%The abstract of your paper
\begin{abstract}
We investigate the non--uniqueness of weak solutions of the Quan\-tum--Hydrodynamic
system. This form of ill--posedness is related to the change of the
number of connected components of the support of the position density
(called nodal domains) of the weak solution throughout its time evolution.
We start by considering a scenario consisting of initial and final
time, showing that if there is a decrease in the number of connected
components, then we have non-uniqueness. This result relies on the
Brouwer invariance of domain theorem. Then we consider the case in
which the results involve a time interval and a full trajectory (position--current
densities). We introduce the concept of trajectory--uniqueness and
its characterization.
\end{abstract}

%The title of your section 1
\section{Introduction}

In this paper, we study the non-uniqueness of (so-called Schr\"{o}\-dinger--generated)
bounded energy weak solutions of the Quantum Hydrodynamic (QHD) system
\cite{landau1979lehrbuch}
\[
\begin{array}{cc}
\begin{array}{c}
\varrho_{t}+\mathrm{div}J=0,\\
J_{t}+\mathrm{div}\left[\frac{J\otimes J}{\varrho}\right]+\varrho\nabla V-\frac{1}{2}\varrho\nabla\left(\frac{\Delta\sqrt{\varrho}}{\sqrt{\varrho}}\right)=0,
\end{array} & x\in\mathbb{R}^{d},\textrm{ }t\in\mathbb{R},\end{array}
\]
along with appropriate initial data, as discussed below. We define
the quantities $\varrho$, $J$ of the QHD system through their connection
with the Schr\"{o}dinger equation
\begin{equation}
i\psi_{t}=-\frac{1}{2}\Delta\psi+V\left(x\right)\psi,\textrm{ }x\in\mathbb{R}^{d},\textrm{ }t\in\mathbb{R},\label{eq:Schr_eq}
\end{equation}
subject to the initial condition
\[
\psi\left(x,t=0\right)=\psi_{0}\left(x\right),\textrm{ }x\in\mathbb{R}^{d}.
\]
Hence, $\varrho=\varrho\left(x,t\right):=\left|\psi\left(x,t\right)\right|^{2}$
and $J=J\left(x,t\right):=\mathrm{Im}\left(\bar{\psi}\left(x,t\right)\nabla\psi\left(x,t\right)\right)$.
We call $\varrho$ the position density and $J$ the current density
generated by the wave function $\psi$. For a detailed derivation
of the QHD system see, e.g., \cite{ancona1989quantum,degond2007isothermal,degond2007quantum,degond2003quantum,jungel2011quasi}.

The Schr\"{o}dinger energy is given by
\[
E\left(t\right):=\frac{1}{2}\int_{\mathbb{R}^{d}}\left|\nabla\psi\left(x,t\right)\right|^{2}dx+\int_{\mathbb{R}^{d}}V\left(x\right)\left|\psi\left(x,t\right)\right|^{2}dx.
\]
It is a constant of the motion, i.e., $E\left(t\right)=E\left(t=0\right)$
$\forall t\in\mathbb{R}$ (under the assumptions made below). 

Henceforth we assume that the potential $V$ is in $C_{loc}^{1,1}\left(\mathbb{R}^{d}\right)$ and bounded from below, such that the
Hamiltonian $H:=-\frac{1}{2}\Delta+V$ is essentially self-adjoint
on $L^{2}\left(\mathbb{R}^{d}\right)$.

Gasser and Markowich \cite{gasser1997quantum} showed that if $E\left(t=0\right)<\infty$
and $\psi_{0}\in L^{2}\left(\mathbb{R}^{d}\right)$, then the position
and current densities corresponding to the uniquely defined energy-conserving
solution $\psi\in C\left(\mathbb{R}_{t};L^{2}\left(\mathbb{R}_{x}^{d}\right)\right)$
satisfy the QHD system in the sense of distributions with initial
data $\varrho_{0}\left(x\right)=\varrho\left(x,t=0\right)=\left|\psi_{0}\left(x\right)\right|^{2}$
and $J_{0}\left(x\right)=J\left(x,t=0\right)=\mathrm{Im}\left(\bar{\psi}_{0}\left(x\right)\nabla\psi_{0}\left(x\right)\right)$.
Clearly we have $\varrho\in C\left(\mathbb{R}_{t};L^{1}\left(\mathbb{R}_{x}^{d}\right)\right)$
and $J\in C\left(\mathbb{R}_{t};\mathcal{D}'\left(\mathbb{R}_{x}^{d}\right)^{d}\right)$,
$J\left(t\right)\in L^{1}\left(\mathbb{R}^{d}\right)^{d}$ $\forall t\in\mathbb{R}$
(uniformly) because of the energy conservation.

Note that the QHD system is formulated above in conservative form,
namely, as compressible Euler equations where the enthalpy is given
by the sum of the so-called Bohm potential 
\[
Q:=-\frac{1}{2}\frac{\Delta\sqrt{\varrho}}{\sqrt{\varrho}}
\]
and the external potential $V$. In this form, vacuum states $\varrho=0$
do not have to be dealt with particularly (see \cite{gasser1997quantum}),
which is not the case for the associated non-conservative form. The
reason for this lies in the fact that the velocity $u:=J/\varrho$
cannot be reasonably defined for wave-functions $\psi$ which exhibit
nodes, i.e., vacuum states where $\psi=0$. Contrary to this, the
internal energy tensor $\left(J\otimes J\right)/\varrho$ makes perfect
sense when $J$ and $\varrho$ are defined through any bounded energy
wave-function (see \cite{antonelli2009finite,antonelli2012quantum}).
However, as we shall point out in the sequel, the conservative form
of QHD is prone to an ill-posedness in the form of non-uniqueness
of finite energy weak solutions, which is generated precisely by the
occurrence of vacuum states.

On the other hand, the QHD equations are the zeroth and first order moment equations of the Bohm equation with the mono kinetic closure. See \cite{gangbo2017bohm,markowich2010bohmian,markowich2012dynamics} for a detailed treatment of this topic.

To motivate our subsequent analysis, take $\psi_{0}=\psi_{0}\left(x\right)$
smooth in $L^{2}\left(\mathbb{R}^{d}\right)$ and define $\varrho_{0}=\varrho_{0}\left(x\right)$
and $J_{0}=J_{0}\left(x\right)$ as before. Moreover, consider a second
$L^{2}-$wave function, $\varphi_{0}=\varphi_{0}\left(x\right)$.
We want to determine the conditions under which 
\begin{equation}
\left|\varphi_{0}\right|^{2}=\varrho_{0},\textrm{ }\mathrm{Im}\left(\bar{\varphi}_{0}\nabla\varphi_{0}\right)=J_{0}\textrm{ }a.e.\textrm{ }on\textrm{ }\mathbb{R}^{d}.\label{eq:cond1}
\end{equation}
To this end, define $D_{0}:=\left\{ x\in\mathbb{R}^{d}:\varrho_{0}\neq0\right\} $
and let $\Lambda_{1}^{0},\Lambda_{2}^{0},\ldots$ be the (countably
many) connected components of $D_{0}$. Furthermore, let $\psi_{0}$
be given by
\[
\psi_{0}\left(x\right)=\sqrt{\varrho_{0}\left(x\right)}\exp\left(iS_{0}\left(x\right)\right),
\]
where $\varrho_{0}\left(x\right)$ and $S_{0}\left(x\right)$ are
smooth real valued functions. Then, 
\[
J_{0}\left(x\right)=\varrho_{0}\left(x\right)\nabla S_{0}\left(x\right).
\]
Set $\varphi_{0}\left(x\right)=\sqrt{\sigma_{0}\left(x\right)}\exp\left(iR_{0}\left(x\right)\right)$
with $\sigma_{0}\left(x\right)=\left|\varphi_{0}\left(x\right)\right|^{2}$.
Thus, (\ref{eq:cond1}) becomes
\[
\begin{array}{c}
\sigma_{0}=\varrho_{0},\\
\varrho_{0}\nabla R_{0}=\varrho_{0}\nabla S_{0},
\end{array}
\]
which implies $\nabla R_{0}=\nabla S_{0}$ on $D_{0}$. Therefore,
we find that due to the connectivity of $\Lambda_{k}^{0}$, $k\in\mathbb{N}$,
\[
R_{0}\left(x\right)=S_{0}\left(x\right)+C_{k},\textrm{ }x\in\Lambda_{k}^{0},
\]
where $C_{k}\in\mathbb{R}$. 

Hence, (\ref{eq:cond1}) is satisfied if and only if 
\[
\varphi_{0}\left(x\right)=\psi_{0}\left(x\right)e^{iC_{k}},\textrm{ }x\in\Lambda_{k}^{0},\textrm{ }k=1,2,\ldots\textrm{ }a.e.\textrm{ in }\mathbb{R}^{d},
\]
for some $C_{k}\in\mathbb{R}$. Note that, under appropriate smoothness
and geometric assumptions, the energy associated to $\varphi_{0}$
is finite for all choices of the constants $C_{k}$ (see Section 2).

We shall argue that, under certain assumptions, the constants $C_{k}$
can be chosen such that the initial wave-function $\varphi_{0}$ generates
a Schr\"{o}dinger solution whose position and current densities differ
from those generated by the Schr\"{o}dinger evolution of $\psi_{0}$ (the
original initial wave-function) at a time $T\neq0$. To be more specific,
consider the following scenario. Let $d=1$ and $\psi_{0}=\psi_{0}\left(x\right)$
smooth with $\psi_{0}\left(0\right)=0$, $\psi_{0}\left(x\right)\neq0$
for $x\neq0$, such that at some $T>0$ we have $\psi\left(x,T\right)\neq0$
for all $x\in\mathbb{R}$ (we shall show an example below). Then $D_{0}=\mathbb{R}-\left\{ 0\right\} $
and $\Lambda_{1}^{0}=\left\{ x<0\right\} $, $\Lambda_{2}^{0}=\left\{ x>0\right\} $.
Define
\[
\varphi_{0}\left(x\right):=\left\{ \begin{array}{ll}
\psi_{0}\left(x\right), & x<0,\\
\psi_{0}\left(x\right)e^{i\alpha}, & x>0,
\end{array}\right.\textrm{for some }\alpha\in\mathbb{R},\textrm{ }\alpha\neq0.
\]
It is straightforward to check that $\varphi_{0}\in H^{1}\left(\mathbb{R}^{d}\right)$.

Assume that the position and current densities of $\psi\left(T\right)$
and $\varphi\left(T\right)$ coincide. Since $\left\{ \left|\psi\left(\cdot,T\right)\right|^{2}\neq0\right\} =\mathbb{R}$,
we conclude that there is $\beta\in\mathbb{R}$ such that 
\[
e^{i\beta}\psi\left(x,T\right)=\varphi\left(x,T\right),\;\forall x\in\mathbb{R}.
\]
Now solve the Schr\"{o}dinger equation
\[
\begin{array}{c}
i\varphi_{t}=-\frac{1}{2}\varphi_{xx}+V\left(x\right)\varphi,\\
\varphi\left(T\right)\textrm{ given }\left(=e^{i\beta}\psi\left(x,T\right)\right),
\end{array}
\]
back to $t=0$ and find that $\varphi\left(x,t=0\right)=e^{i\beta}\psi\left(x,t=0\right)$
on $\mathbb{R}$, i.e., $\varphi_{0}\left(x\right)=e^{i\beta}\psi_{0}\left(x\right)$
on $\mathbb{R}$, which is a contradiction. Therefore, we conclude
the non-uniqueness of the corresponding initial value problem (IVP)
for the QHD system.

As an example, consider the harmonic oscillator in 1D, i.e., equation
(\ref{eq:Schr_eq}) with $V\left(x\right)=x^{2}/2$, $x\in\mathbb{R}$.
In this case there is a family of solutions (with appropriate initial
conditions) for $n=0,1,2,\ldots$ and energy levels given by $E_{n}=\left(n+1/2\right)$.
The first three members of the family are
\[
\psi_{0}\left(x,t\right)=e^{-iE_{0}t}e^{-x^{2}/2}=e^{-it/2}e^{-x^{2}/2},
\]
\[
\psi_{1}\left(x,t\right)=xe^{-\frac{3}{2}it}e^{-x^{2}/2},
\]
\[
\psi_{2}\left(x,t\right)=\left(1-2x^{2}\right)e^{-\frac{5}{2}it}e^{-x^{2}/2}.
\]
Due to the linearity of the Schr\"{o}dinger equation, any linear combination
of the previous solutions will also be a solution. Hence, consider
\[
\psi\left(x,t\right)=\psi_{0}\left(x,t\right)-\psi_{2}\left(x,t\right)=e^{-i\frac{t}{2}-\frac{x^{2}}{2}}\left(1-\left(1-2x^{2}\right)e^{-2it}\right).
\]
Then
\[
\psi\left(x,t\right)=e^{-i\frac{t}{2}-\frac{x^{2}}{2}}\left(1-\left(1-2x^{2}\right)\cos2t+i\left(1-2x^{2}\right)\sin2t\right)
\]
with
\[
\left|\psi\left(x,t\right)\right|^{2}=e^{-x^{2}}\left(\left(1-\left(1-2x^{2}\right)\cos2t\right)^{2}+\left(1-2x^{2}\right)^{2}\left(\sin2t\right)^{2}\right).
\]
Clearly, $\left|\psi\left(x,t=0\right)\right|^{2}=4e^{-x^{2}}x^{4}=0$
if and only if $x=0$. In general, $\left|\psi\left(x,t\right)\right|^{2}=0$
if and only if $\left(1-\left(1-2x^{2}\right)\cos2t\right)^{2}=0$
and $\left(1-2x^{2}\right)^{2}\left(\sin2t\right)^{2}=0$. Therefore,
$\left|\psi\left(x,t\right)\right|^{2}=0$ if and only if
\[
x=0\textrm{ and }t=\frac{l\pi}{2},\textrm{ }l\textrm{ even,}
\]
\[
\textrm{or}
\]
\[
x=\pm1\textrm{ and }t=\frac{l\pi}{2},\textrm{ }l\textrm{ odd.}
\]
Thus, we have our required scenario if we choose $\psi_{0}=\psi(x,t=0)$
and any $0<T<\pi/2$.

The primary goal of this paper is to give rather explicit (sufficient)
conditions on a bounded energy Schr\"{o}dinger solution which guarantee
that the QHD trajectory (position-current densities) generated by
such wave-function is not unique in the sense that a different QHD
trajectory intersects it at some $t\in\mathbb{R}$. It turns out that
we can state such conditions in connection with the topological structure
of the so-called nodal domains of the wave function, defined as the
connected components of the set where the wave function does not vanish
(in other words, the connected domains of the non-vacuum set of the
quantum flow).

The rest of the paper is organized as follows. In Section 2 we generalize
the previous result to wave function solution of the Schr\"{o}dinger IVP
(QHD-IVP) with less regularity, arbitrary dimension, and an arbitrary
number of connected components. In Section 3, we consider the case
in which the results involve a time interval and a full QHD trajectory.

\section{(Non) Uniqueness of the IVP}
\begin{proposition}
Let $\psi$ be a realization of an $H^{1}\left(\mathbb{R}^{d}\right)-$function,
$\varrho:=\left|\psi\right|^{2}$, and $J=\mathrm{Im}\left(\bar{\psi}\nabla\psi\right)$.
Let $\Lambda$ be a connected component of the (set-theoretic) support
of $\varrho$, defined by $D_{0}:=\left\{ x\in\mathbb{R}^{d}:\varrho\left(x\right)>0\right\} $,
and assume that $\Lambda$ is open in $\mathbb{R}^{d}$. Let $\varphi\in H^{1}\left(\mathbb{R}^{d}\right)$,
$\sigma:=\left|\varphi\right|^{2}$ and $I:=\mathrm{Im}\left(\bar{\varphi}\nabla\varphi\right)$.
If
\[
\varphi=\sigma,\;J=I,\textrm{ }a.e.\textrm{ }on\textrm{ }\Lambda
\]
then there exists a real constant, $C$, such that
\[
\psi\left(x\right)=\varphi\left(x\right)e^{iC},\;x\in\Lambda\;a.e.
\]
\end{proposition}

\begin{proof}
Clearly, $\varrho$, $\sigma$, $I$, $J\in L^{1}\left(\mathbb{R}^{d}\right)$.
Consider, for $\omega\in C_{0}^{\infty}\left(\Lambda\right)$, $\omega\geq0$:
\[
\begin{array}{l}
\int_{\Lambda}\left|I-J\right|\omega dx\\
=\int_{\Lambda}\left|\mathrm{Im}\left(\bar{\varphi}\nabla\varphi-\bar{\psi}\nabla\psi\right)\right|\omega dx\\
=\int_{\Lambda}\left|\mathrm{Im}\left(\frac{\nabla\varphi}{\varphi}-\frac{\nabla\psi}{\psi}\right)\right|\omega\varrho dx\\
=\int_{\Lambda}\left|\mathrm{Im}\left(\nabla\ln\frac{\varphi}{\psi}\right)\right|\omega\varrho dx\\
=\int_{\Lambda}\left|\nabla\mathrm{Im}\left(\ln\frac{\varphi}{\psi}\right)\right|\omega\varrho dx.
\end{array}
\]
Hence,
\[
0=\nabla\mathrm{Im}\left(\ln\left(\frac{\varphi}{\psi}\right)\right)\textrm{ }a.e.\textrm{ on }\Lambda.
\]
Therefore,
\[
\arg\psi=\arg\varphi+C+2k\pi.
\]

We have
\[
\psi=\sqrt{\varrho}e^{i\arg\psi},\textrm{ }\varphi=\sqrt{\sigma}e^{i\arg\varphi},
\]
and since $\varrho=\sigma$, we conclude that
\[
\psi=\varphi e^{iC_{\alpha}}\textrm{ }a.e.\textrm{ on }\Lambda.
\]
\end{proof}
\begin{remark}
Assume that all the connected components $\Lambda_{\alpha}$ of $D_{0}$
are of locally finite perimeter ($\Lambda_{\alpha}$ are Caccioppoli
sets), that $\psi$ is continuous on $\mathbb{R}^{d}$ and in $H^{1}\left(\mathbb{R}^{d}\right)$.
Then the characteristic function $\mathbf{1}_{\Lambda_{\alpha}}$
of $\Lambda_{\alpha}$ has locally bounded total variation, which
in turn implies that its distributional gradient is a vector valued
(signed) Radon measure supported on the boundary of $\Lambda_{\alpha}$.
Thus the function $\varphi$, defined by
\[
\varphi=\psi\sum_{\alpha\in A}e^{-iC_{\alpha}}\mathbf{1}_{\Lambda_{\alpha}},
\]
is in $H^{1}\left(\mathbb{R}^{d}\right)$, and we have $\varrho=\sigma$
and $J=I$ on $\mathbb{R}^{d}$. Specifically, we compute 
\[
\nabla\varphi=\nabla\psi\sum_{\alpha\in A}e^{-iC}\mathbf{1}_{\Lambda_{\alpha}},
\]
taking into account that $\left.\psi\right|_{\partial\Lambda_{\alpha}}=0$.
Note that the index set $A\subseteq\mathbb{N}$.
\end{remark}

\begin{proposition}
Let $\psi=\psi\left(x,t\right)$ be a bounded energy solution of the
Schrödinger equation on $\left[0,T\right]$ with $\psi$ continuous
(in $x$) for $t=0,T$. Define
\[
\begin{array}{c}
N_{0}\in\mathbb{N}:=\textrm{ Number of connected components of }\left\{ \varrho\left(\cdot,t=0\right)\neq0\right\} :=D_{0},\\
N_{T}\in\mathbb{N}:=\textrm{ Number of connected components of }\left\{ \varrho\left(\cdot,t=T\right)\neq0\right\} :=D_{T},
\end{array}
\]
and assume $N_{T}<N_{0}\leq\infty$. Furthermore, assume that all
the connected components of $D_{0}$ are of locally finite perimeter.
Then there is non-uniqueness of Schrödinger-generated bounded energy
solutions of the QHD system on $\left[0,T\right]$, with $\varrho\left(\cdot,t=0\right)$,
$J\left(\cdot,t=0\right)$ given. 
\end{proposition}

\begin{proof}
To simplify the presentation, assume that $N_{0}<\infty$ . Let $\Lambda_{1}^{0},\ldots,\Lambda_{N_{0}}^{0}$
be the connected components of $D_{0}$ and let $\Lambda_{1}^{T},\ldots,\Lambda_{N_{T}}^{T}$
be the connected components of $D_{T}$. 

Note that $\psi\left(\cdot,t=0\right)$ and
\[
\varphi\left(x,t=0\right):=\psi\left(x,t=0\right)e^{i\alpha_{l}},\;x\in\Lambda_{l}^{0},\;l=1,\ldots,N_{0}
\]
are in $H^{1}\left(\mathbb{R}^{d}\right)$ and generate the same position
and current densities at $t=0$ for all choices $\alpha_{l}\in\mathbb{R}$. 

Assume that the QHD system has a unique Schrödinger-generated bounded
energy solution on $\left[0,T\right]$. Then, at $t=T$, we must have
\[
\varphi\left(x,t=T\right)=\psi\left(x,t=T\right)e^{i\beta_{k}},\;x\in\Lambda_{k}^{T},\;k=1,\ldots,N_{T},
\]
where $\varphi\left(t\right)$ is the Schrödinger solution with initial
datum $\varphi\left(\cdot,t=0\right)$ and $\varrho=\left|\varphi\right|^{2}$,
$J=\mathrm{Im}\left(\bar{\varphi}\nabla\varphi\right)$. This defines
a map $\mathbb{T}^{N_{0}}\mapsto\mathbb{T}^{N_{T}}$ ($\mathbb{T}^{n}$
is the $n-$dimensional torus)
\[
F:\left(\alpha_{1},\ldots,\alpha_{N_{0}}\right)\rightarrow\left(\beta_{1},\ldots,\beta_{N_{T}}\right).
\]
This map is injective since the backward Schrödinger IVP has the uniqueness
property. Continuity of $F$ is a consequence the $L^{2}\left(\mathbb{R}^{d}\right)-$continuity
of the Schrödinger evolution in the following way. Take a sequence
$\left(\alpha_{1}^{\left(m\right)},\ldots,\alpha_{N_{0}}^{\left(m\right)}\right)\rightarrow\left(\alpha_{1}^{\left(0\right)},\ldots,\alpha_{N_{0}}^{\left(0\right)}\right)$
as $m\rightarrow\infty$. Hence,
\[
\begin{array}{c}
\sum_{l=1}^{N_{0}}\int_{\Lambda_{l}^{0}}\left|\psi\left(x,t=0\right)e^{i\alpha_{l}^{\left(m\right)}}-\psi\left(x,t=0\right)e^{i\alpha_{l}^{\left(0\right)}}\right|^{2}dx\\
=\sum_{l=1}^{N_{0}}\int_{\Lambda_{l}^{0}}\left|\psi\left(x,t=0\right)\right|^{2}\left|1-e^{i\left(\alpha_{l}^{\left(0\right)}-\alpha_{l}^{\left(m\right)}\right)}\right|^{2}dx\rightarrow0.
\end{array}
\]
Therefore, 
\[
\psi^{\left(m\right)}\left(x,t=0\right):=\psi\left(x,t=0\right)e^{i\alpha_{l}^{\left(m\right)}},\;x\in\Lambda_{l}^{0},\;l=1,\ldots,N_{0}
\]
converges to
\[
\psi^{\left(0\right)}\left(x,t=0\right):=\psi\left(x,t=0\right)e^{i\alpha_{l}^{\left(0\right)}},\;x\in\Lambda_{l}^{0},\;l=1,\ldots,N_{0}
\]
in $L^{2}\left(\mathbb{R}^{d}\right)$. By $L^{2}-$continuity of
the Schrödinger evolution, we have that
\[
\psi^{\left(m\right)}\left(x,t=T\right):=\psi\left(x,t=T\right)e^{i\beta_{k}^{\left(m\right)}},\;x\in\Lambda_{k}^{T},\;k=1,\ldots,N_{T}
\]
converges in $L^{2}\left(\mathbb{R}^{d}\right)$ to
\[
\psi^{\left(0\right)}\left(x,t=T\right):=\psi\left(x,t=T\right)e^{i\beta_{k}^{\left(0\right)}},\;x\in\Lambda_{k}^{T},\;k=1,\ldots,N_{T}.
\]
By the same computation, we find that
\[
\begin{array}{c}
\int_{\mathbb{R}^{d}}\left|\psi^{\left(0\right)}\left(x,T\right)-\psi^{\left(m\right)}\left(x,T\right)\right|^{2}dx\\
=\sum_{k=1}^{N_{T}}\int_{\Lambda_{k}^{T}}\left|\psi\left(x,T\right)\right|^{2}dx\left|1-e^{i\left(\beta_{k}^{\left(0\right)}-\beta_{k}^{\left(m\right)}\right)}\right|^{2}\rightarrow0
\end{array}
\]
as $m\rightarrow\infty$ if and only if $\beta_{k}^{\left(m\right)}\rightarrow\beta_{k}^{\left(0\right)},$
$k=1,\ldots,N_{T}$. Hence, $F$ is continuous. 

Since by assumption $N_{0}>N_{T}$, we have a contradiction due to
the fact that, as a consequence of the Brouwer invariance of domain
theorem (see, e.g., \cite{dold2012lectures,hatcher2002algebraic,massey1991basic}),
there is no continuous injective mapping from $\mathbb{R}^{n}$ to
$\mathbb{R}^{m}$ when $n>m$. 

It is straightforward to show that the result also holds for $N_{0}=\infty>N_{T}$.
\end{proof}

\section{Trajectory case}

Let $\psi\in C\left(\mathbb{R}_{t};L^{2}\left(\mathbb{R}_{x}^{d}\right)\right)\cap L^{\infty}\left(\mathbb{R}_{t};H^{1}\left(\mathbb{R}_{x}^{d}\right)\right)$
be a bounded energy solution of the Schrödinger equation (\ref{eq:Schr_eq}).
Let $\varrho:=\left|\psi\right|^{2}$ and $J:=\mathrm{Im}\left(\bar{\psi}\nabla\psi\right)$,
and assume that $\varrho\in C\left(\mathbb{R}_{x}^{d}\times\mathbb{R}_{t}\right)$.
Denote
\[
\Omega:=\left\{ \left(x,t\right)\in\mathbb{R}^{d+1}:\varrho\left(x,t\right)\neq0\right\} ,
\]
and let $\Omega_{\alpha}\subseteq\Omega$, $\alpha\in B\subseteq\mathbb{N}$,
be the connected components of $\Omega$. Note that each set $\Omega_{\alpha}$
is open in $\mathbb{R}^{d+1}$ and 
\[
\partial\Omega_{\alpha}\subseteq\left\{ \left(x,t\right)\in\mathbb{R}^{d+1}:\varrho\left(x,t\right)=0\right\} .
\]
\begin{proposition}
Let $\varphi\in C\left(\mathbb{R}_{t};L^{2}\left(\mathbb{R}_{x}^{d}\right)\right)\cap L^{\infty}\left(\mathbb{R}_{t};H^{1}\left(\mathbb{R}_{x}^{d}\right)\right)$
be another bounded energy solution of the Schrödinger equation (\ref{eq:Schr_eq})
and denote $\sigma:=\left|\varphi\right|^{2}$, $I:=\mathrm{Im}\left(\bar{\varphi}\nabla\varphi\right)$.
If
\[
\varrho=\sigma,\textrm{ }J=I\textrm{ }\forall t\textrm{ }a.e.\textrm{ }in\textrm{ }\mathbb{R}_{x}^{d},
\]
then there exists a constant, $C_{\alpha}\in\mathbb{R}$, for every
$\alpha\in B$ such that $\varphi\left(x,t\right)=\psi\left(x,t\right)e^{iC_{\alpha}}$
a.e. in $\Omega_{\alpha}$.
\end{proposition}

\begin{proof}
Let $\omega\in C_{0}^{\infty}\left(\Omega_{\alpha}\right)$ and consider
\[
\begin{array}{l}
\int_{\mathbb{R}_{t}}\int_{\mathbb{R}_{x}^{d}}\left|J-I\right|\omega dx\\
=\int_{\mathbb{R}_{t}}\int_{\mathbb{R}_{x}^{d}}\left|\mathrm{Im}\left(\bar{\psi}\nabla\psi-\bar{\varphi}\nabla\varphi\right)\right|\omega dx\\
=\int_{\mathbb{R}_{t}}\int_{\mathbb{R}_{x}^{d}}\left|\mathrm{Im}\left(\frac{\nabla\psi}{\psi}-\frac{\nabla\varphi}{\varphi}\right)\right|\omega\varrho dx\\
=\int_{\mathbb{R}_{t}}\int_{\mathbb{R}_{x}^{d}}\left|\nabla_{x}\mathrm{Im}\ln\left(\frac{\psi}{\varphi}\right)\right|\omega\varrho dx=0.
\end{array}
\]
Now let $\omega$ be supported in an arbitrary convex subset, $\Sigma_{\alpha}$,
of $\Omega_{\alpha}$. Then
\[
\arg\varphi=\arg\psi+C_{\Sigma_{\alpha}}\left(t\right)\textrm{ in }\Sigma_{\alpha}
\]
for some measurable function $C_{\Sigma_{\alpha}}=C_{\Sigma_{\alpha}}\left(t\right)$
on $\mathbb{T}^{d}$. This gives
\[
\varphi=\psi e^{iC_{\Sigma_{\alpha}}\left(t\right)}\textrm{ in }\Sigma_{\alpha}.
\]
We insert into the Schrödinger equation and find ($\psi\neq0$, $\varphi\neq0$
in $\Omega_{\alpha}$)
\[
\dot{C}_{\Sigma_{\alpha}}\left(t\right)=0\textrm{ for all }t\in\Sigma'_{\alpha}=\Sigma_{\alpha}\cap\left\{ \left(y,t\right):y\in\mathbb{R}^{d}\right\} \neq\emptyset.
\]
Therefore for every convex open subset $\Sigma_{\alpha}$ of $\Omega_{\alpha}$
there exists $C_{\Sigma_{\alpha}}\in\mathbb{R}$ such that
\[
\varphi=\psi e^{iC_{\Sigma_{\alpha}}}\textrm{ in }\Sigma_{\alpha}.
\]
To prove that $C_{\Sigma_{\alpha}}$ depends only on $\alpha$ and
not on the convex subset of $\Omega_{\alpha}$ we take any two points
$\left(x_{1},t_{1}\right),\left(x_{2},t_{2}\right)\in\Omega_{\alpha}$
and connect them by a continuous curve, $\Gamma\subseteq\Omega_{\alpha}$.
Now cover $\Gamma$ by (finitely many, since $\Gamma$ is compact)
open balls $B_{\alpha,1},\ldots,B_{\alpha,L}\subseteq\Omega_{\alpha}$.
In each ball $B_{\alpha,l}$ we have
\[
\varphi=\psi e^{iC_{B_{\alpha,l}}}.
\]
Since for each $B_{\alpha,l_{1}}$ there exists $B_{\alpha,l_{2}}$
with $l_{1}\neq l_{2}$ such that
\[
B_{\alpha,l_{1}}\cap B_{\alpha,l_{2}}\neq\emptyset,
\]
we conclude that $C_{B_{\alpha,1}}=\cdots=C_{B_{\alpha,l}}.$ Therefore,
a unique constant $C=C_{\Omega_{\alpha}}$ exists such that
\[
\varphi=\psi e^{iC_{\Omega_{\alpha}}}\textrm{ on }\Omega_{\alpha}\textrm{ }a.e.
\]
\end{proof}
\begin{remark}
Let $C_{\alpha}\in\mathbb{R}$ for $\alpha\in B$, $\psi\in C\left(\mathbb{R}_{x}^{d}\times\mathbb{R}_{t}\right)\cap L^{\infty}\left(\mathbb{R}_{t};H^{1}\left(\mathbb{R}^{d}\right)\right)$.
Then the function
\[
\varphi\left(x,t\right)=\psi\left(x,t\right)\sum_{\alpha\in B}e^{iC_{\alpha}}\mathbf{1}_{\Omega_{\alpha}}
\]
is in $L^{\infty}\left(\mathbb{R}_{t};H^{1}\left(\mathbb{R}^{d}\right)\right)$
if all sets $\Omega_{\alpha}$ have locally finite perimeter. Actually,
it suffices to ask that $\nabla_{x}\mathbf{1}_{\Omega_{\alpha}}\in L_{loc}^{\infty}\left(\mathbb{R}_{t};\mathcal{M}\left(\mathbb{R}_{x}^{d}\right)^{d}\right)$,
where $\mathcal{M}\left(\mathbb{R}_{x}^{d}\right)$ is the set of
(scalar signed) Radon measures on $\mathbb{R}_{x}^{d}$.
\end{remark}

\begin{definition}
Let
\[
\left\{ \left(\varrho\left(t\right),J\left(t\right)\right):t\in\mathbb{R}\right\} \in C\left(\mathbb{R}_{t};L^{1}\left(\mathbb{R}^{d}\right)\right)\times\left(L^{\infty}\left(\mathbb{R}_{t};L^{1}\left(\mathbb{R}^{d}\right)\right)^{d}\cap C\left(\mathbb{R}_{t};\mathcal{D}'\left(\mathbb{R}_{x}^{d}\right)\right)^{d}\right)
\]
be a solution-curve of the QHD system for $-\infty<t<\infty$. It
is called trajectory-unique if for all other such solution curves
$\left\{ \left(\sigma\left(t\right),I\left(t\right)\right):t\in\mathbb{R}\right\} $
we have $\left(\varrho,J\right)\left(t\right)\neq\left(\sigma,I\right)\left(t\right)$
$\forall t\in\mathbb{R}$ or $\left(\varrho,J\right)\left(t\right)=\left(\sigma,I\right)\left(t\right)$
$\forall t\in\mathbb{R}$ .

In simple terms, this means that different trajectories do not intersect,
neither forward nor backward in time. 
\end{definition}

\begin{proposition}
Consider a QHD-trajectory, $S=\left\{ \left(\varrho\left(t\right),J\left(t\right)\right):t\in\mathbb{R}\right\} $,
with $\psi\in C\left(\mathbb{R}_{x}^{d}\times\mathbb{R}_{t}\right)$
and bounded energy. If for some $T\in\mathbb{R}$ the number $N$
of connected components of 
\[
\mathbb{R}^{d+1}\supseteq\Omega:=\left\{ \varrho\left(x,t\right)\neq0:\left(x,t\right)\in\mathbb{R}^{d+1}\right\} 
\]
 is smaller than the number $N_{T}$ of connected components of 
\[
\mathbb{R}^{d}\supseteq\Lambda_{T}:=\left\{ \varrho\left(x,T\right)\neq0:x\in\mathbb{R}^{d}\right\} ,
\]
then $S$ is not trajectory-unique, provided that all connected components
have locally finite perimeter. 
\end{proposition}

\begin{proof}
Let $N<N_{T}<\infty$ and assume that $S$ is $generated$ by the
wave-function $\psi=\psi\left(x,t\right)$, i.e., $\varrho=\left|\psi\right|^{2}$,
$J=\mathrm{Im}\left(\bar{\psi}\nabla\psi\right)$ and $\psi$ solves
the Schrödinger equation. 

Note that
\[
\varphi\left(x,t\right)=\psi\left(x,t\right)e^{iC_{\Omega_{\alpha}}},\textrm{ }\left(x,t\right)\in\Omega_{\alpha},
\]
where $\Omega_{\alpha},$ $\alpha=1,\ldots,N$, are the connected
components of $\Omega$, generates the same trajectory $S$ (sufficient
and necessary). Now, for $C_{\Lambda_{T,1}},\ldots,C_{\Lambda_{T,N_{T}}}\in\mathbb{R}$
set
\[
\sigma\left(x,T\right)=\psi\left(x,T\right)e^{iC_{\Lambda_{T,\beta}}},\textrm{ }x\in\Lambda_{T,\beta},
\]
where $\Lambda_{T,1},\ldots,\Lambda_{T,N_{T}}$ are the connected
components of $\Lambda_{T}$. Solve the IVP for the Schrödinger equation
starting from $\sigma\left(\cdot,T\right)$ backwards and forwards.
If all the possible choices of $\sigma\left(T\right)$ gave the same
trajectory $S$, there would have to exist a continuous and injective
map from $\mathbb{R}^{N_{T}}\mapsto\mathbb{R}^{N}$, which is a contradiction
since $N_{T}>N$. The proof extend easily to $N_{T}=\infty$. 
\end{proof}
\begin{proposition}
Let $\psi=\psi\left(x,t\right)$ be a bounded energy solution of the
Schrödinger equation with $\varrho\left(\cdot,t=T_{1}\right)$ and
$\varrho\left(\cdot,t=T_{2}\right)$ continuous on $\mathbb{R}^{d}$
($T_{1}<T_{2}$) and let all connected components of $\left\{ \varrho\left(\cdot,t=T_{1}\right)\neq0\right\} $
and $\left\{ \varrho\left(\cdot,t=T_{2}\right)\neq0\right\} $ be
Caccioppoli sets. If $N_{T_{1}}\neq N_{T_{2}}$, then $\left\{ \left(\varrho,J\right):t\in\left[T_{1},T_{2}\right]\right\} $
(QHD-trajectory) is not trajectory-unique.
\end{proposition}

\begin{proof}
Already established for $N_{T_{1}}>N_{T_{2}}$ (Proposition 2). If
$N_{T_{2}}>N_{T_{1}}$ apply the same argument calculating backwards
in time.
\end{proof}
Now let $\Omega_{1}\neq\Omega_{2}$ be two adjacent connected components
of $\Omega\subseteq\mathbb{R}_{x}^{d}\times\mathbb{R}_{t}$, with
the smooth interface surface $\Gamma=\partial\Omega_{1}\cap\partial\Omega_{2}$.

Let $\psi=\psi\left(x,t\right)$ be a smooth solution of the Schrödinger
equation, with obviously $\left.\psi\right|_{\Gamma}=0$. Denote by
$\varUpsilon$ the unit normal of $\Gamma$, pointing (for definiteness
sake) into $\Omega_{1}$ and set $\varUpsilon=\left(\begin{array}{c}
\varUpsilon_{x}\\
\varUpsilon_{t}
\end{array}\right)$ according to the coordinate ordering $\left(\begin{array}{c}
x\\
t
\end{array}\right)$. 
\begin{proposition}
Let $C_{1}\neq C_{2}$ be real constants and set
\[
\varphi\left(x,t\right)=\psi\left(x,t\right)\cdot\left\{ \begin{array}{cc}
e^{iC_{1}}, & \left(x,t\right)\in\Omega_{1},\\
e^{iC_{2}}, & \left(x,t\right)\in\Omega_{2}.
\end{array}\right.
\]
Then $\varphi$ is a solution of the Schrödinger equation in $\Omega_{1}\cup\Gamma\cup\Omega_{2}$
if and only if
\[
\left.\nabla\psi\cdot\varUpsilon_{x}\right|_{\Gamma}=0.
\]
\end{proposition}

\begin{proof}
Set $\alpha_{1}=e^{iC_{1}}$, $\alpha_{2}=e^{iC_{2}}$. Then
\[
\varphi=\alpha_{1}\psi\mathbf{1}_{\Omega_{1}}+\alpha_{2}\psi\mathbf{1}_{\Omega_{2}}\textrm{ in }\Omega_{1}\cup\Gamma\cup\Omega_{2}.
\]
Compute, in $\Omega_{1}\cup\Gamma\cup\Omega_{2}$
\[
\varphi_{t}=\alpha_{1}\psi_{t}\mathbf{1}_{\Omega_{1}}+\alpha_{2}\psi_{t}\mathbf{1}_{\Omega_{2}}+\alpha_{1}\psi\frac{\partial}{\partial t}\mathbf{1}_{\Omega_{1}}+\alpha_{2}\psi\frac{\partial}{\partial t}\mathbf{1}_{\Omega_{2}}.
\]
Note that $\left(\frac{\partial}{\partial t}\mathbf{1}_{\Omega_{i}}\right)$
is supported on $\partial\Omega_{i}$, where $\psi$ vanishes continuously.
Thus
\[
\psi\frac{\partial}{\partial t}\mathbf{1}_{\Omega_{i}}=0
\]
and
\[
\varphi_{t}=\alpha_{1}\psi_{t}\mathbf{1}_{\Omega_{1}}+\alpha_{2}\psi_{t}\mathbf{1}_{\Omega_{2}}.
\]
The analogous computation holds with $\nabla\mathbf{1}_{\Omega_{i}}$.
Then
\[
\nabla\varphi=\alpha_{1}\nabla\psi\mathbf{1}_{\Omega_{1}}+\alpha_{2}\nabla\psi\mathbf{1}_{\Omega_{2}}
\]
and
\[
\begin{array}{cl}
\Delta\varphi= & \alpha_{1}\Delta\psi\mathbf{1}_{\Omega_{1}}+\alpha_{2}\Delta\psi\mathbf{1}_{\Omega_{2}}\\
 & +\alpha_{1}\nabla\psi\cdot\nabla\mathbf{1}_{\Omega_{1}}+\alpha_{2}\nabla\psi\cdot\nabla\mathbf{1}_{\Omega_{2}}.
\end{array}
\]
Plug into the Schrödinger equation to find (since $\psi$ solves the
Schrödinger equation)
\[
0=\nabla\psi\cdot\left(\alpha_{1}\nabla\mathbf{1}_{\Omega_{1}}+\alpha_{2}\nabla\mathbf{1}_{\Omega_{2}}\right).
\]
Take a test-function, $\omega=\omega\left(x,t\right)\in C_{0}^{\infty}\left(\mathbb{R}_{x}^{d}\times\mathbb{R}_{t}\right)$:
\[
\begin{array}{cl}
0 & =\left\langle \alpha_{1}\nabla\psi\cdot\nabla\mathbf{1}_{\Omega_{1}}+\alpha_{2}\nabla\psi\cdot\nabla\mathbf{1}_{\Omega_{2}},\omega\right\rangle _{x,t}\\
 & =-\alpha_{1}\int_{\Omega_{1}}\mathrm{div}_{x}\left(\omega\nabla\psi\right)dxdt-\alpha_{2}\int_{\Omega_{2}}\mathrm{div}_{x}\left(\omega\nabla\psi\right)dxdt\\
 & =\alpha_{1}\int_{\Gamma}\omega\nabla\psi\cdot\varUpsilon_{x}ds-\alpha_{2}\int_{\Gamma}\omega\nabla\psi\cdot\varUpsilon_{x}ds\\
 & =\left(\alpha_{1}-\alpha_{2}\right)\int_{\Gamma}\nabla\psi\cdot\varUpsilon_{x}\omega ds.
\end{array}
\]
Therefore,
\[
\left.\nabla\psi\cdot\varUpsilon_{x}\right|_{\Gamma}=0.
\]
\end{proof}
For the following results, we use the generalized Green's formula
by De Giorgi-Federer:
\[
\int_{E}\mathrm{div}Fdx=\int_{\partial^{*}E}F\cdot\varUpsilon d\mathcal{H}^{n-1},
\]
where $F$ is any locally Lipschitz function, $E\subseteq\mathbb{R}^{n}$
is of finite perimeter, $\varUpsilon$ is the normal of $\partial E$,
and $\partial^{*}E$ is the reduced boundary (see, e.g., \cite{evans2015measure,maggi2012sets}).
\begin{proposition}
Let $\Omega_{1},\Omega_{2}\subseteq\mathbb{R}_{x}^{d}\times\mathbb{R}_{t}$
be of locally finite perimeter. Let $\nabla\psi\in C_{loc}^{0,1}\left(\mathbb{R}_{x}^{d}\times\mathbb{R}_{t}\right)$.
Then $\varphi$ (defined as in Proposition 6) solves the Schrödinger
equation if and only if
\[
\nabla\psi\cdot\varUpsilon_{x}=0\textrm{ }\mathcal{H}^{d}-a.e.\textrm{ }on\textrm{ }\left(\partial\Omega_{1}\cap\partial\Omega_{2}\right)^{*}.
\]
\end{proposition}

The consequence is:
\begin{proposition}
Let all connected components $\Omega_{\alpha}$ of $\Omega$ ($\alpha\in B$)
be of locally finite perimeter and let
\[
\nabla\psi\in C_{loc}^{0,1}\left(\mathbb{R}_{x}^{d}\times\mathbb{R}_{t}\right).
\]
Assume that there is a non-space-like interface segment, $\Gamma$,
between two connected components such that
\[
\nabla\psi\cdot\varUpsilon_{x}\neq0\textrm{ }\mathcal{H}^{d}-a.e.\textrm{ }on\textrm{ }\Gamma^{*}.
\]
Then the QHD-trajectory generated by $\psi=\psi\left(x,t\right)$
is not trajectory-unique on $\mathbb{R}_{t}$.
\end{proposition}

\begin{proof}
Choose $T\in\mathbb{R}$ such that
\[
\Gamma\cap\left\{ \left(y,T\right):y\in\mathbb{R}^{d}\right\} \neq\emptyset.
\]
Choose $C_{1}\neq C_{2}$ real numbers and set
\[
\varphi\left(x,T\right):=\left\{ \begin{array}{l}
\psi\left(x,T\right)e^{iC_{1}}\textrm{ }in\textrm{ }\Omega_{1}\cap\left\{ t=T\right\} ,\\
\psi\left(x,T\right)e^{iC_{2}}\textrm{ }in\textrm{ }\Omega_{2}\cap\left\{ t=T\right\} ,\\
\psi\left(x,T\right)\textrm{ }elsewhere,
\end{array}\right.
\]
and solve the Schrödinger IVP with $\varphi\left(\cdot,T\right)$
as initial datum. By the above
\[
\varphi\left(x,t\right)\neq\psi\left(x,t\right)\cdot\left\{ \begin{array}{cc}
e^{iC_{1}}, & \left(x,t\right)\in\Omega_{1},\\
e^{iC_{2}}, & \left(x,t\right)\in\Omega_{2},
\end{array}\right.
\]
and hence trajectory non-uniqueness follows.
\end{proof}

\section*{Acknowledgements}

The authors acknowledge in-depth discussions with Paolo Antonelli
and Pierangelo Marcati on the mathematical analysis of the QHD system.

% You may incorporate your references as follows in your main tex file.
% Using BibTex is not recommended but can be handled.

\medskip
% The data information below will be filled by AIMS editorial staff
Received xxxx 20xx; revised xxxx 20xx.
\medskip


\begin{thebibliography}{99}

\bibitem{ancona1989quantum}
  \newblock M. G. Ancona and G. J. Iafrate,
  \newblock Quantum correction to the equation of state of an electron gas in a semiconductor,
  \newblock \emph{Physical Review B}, \textbf{39} (1989), 9536--9540.


\bibitem{antonelli2009finite}
  \newblock P. Antonelli and P. Marcati,
  \newblock On the finite energy weak solutions to a system in quantum fluid dynamics,
  \newblock \emph{Communications in Mathematical Physics}, \textbf{287} (2009), 657--686.
  

\bibitem{antonelli2012quantum}
  \newblock P. Antonelli and P. Marcati,
  \newblock The quantum hydrodynamics system in two space dimensions,
  \newblock \emph{Archive for Rational Mechanics and Analysis}, \textbf{203} (2012), 499--527.


\bibitem{degond2007isothermal}
  \newblock P. Degond, S. Gallego and F. M{\'e}hats,
  \newblock Isothermal quantum hydrodynamics: derivation, asymptotic analysis, and simulation,
  \newblock \emph{Multiscale Modeling \& Simulation}, \textbf{6} (2007), 246--272.


\bibitem{degond2007quantum}
  \newblock P. Degond, S. Gallego and F. M{\'e}hats,
  \newblock On quantum hydrodynamic and quantum energy transport models,
  \newblock \emph{Communications in Mathematical Sciences}, \textbf{5} (2007), 887--908.
  
  
\bibitem{degond2003quantum}
  \newblock P. Degond and C. Ringhofer,
  \newblock Quantum moment hydrodynamics and the entropy principle,
  \newblock \emph{Journal of Statistical Physics}, \textbf{112} (2003), 587--628.
  
  
\bibitem{dold2012lectures}
  \newblock A. Dold,
  \newblock \emph{Lectures on Algebraic Topology},
  \newblock Springer Science \& Business Media, 2012.
  
    
\bibitem{evans2015measure}
  \newblock L. C. Evans and R. F. Gariepy,
  \newblock \emph{Measure Theory and Fine Properties of Functions},
  \newblock CRC press, 2015.  

\bibitem{gangbo2017bohm}
  \newblock W. Gangbo, J. Haskovec, P. Markowich and J. Sierra,
  \newblock An optimal transport approach for the kinetic bohmian equation,
  \newblock \emph{Zapiski Nauchnyh Seminarov POMI}, \textbf{457} (2017), 114--167.    
  
\bibitem{gasser1997quantum}
  \newblock I. Gasser and P. Markowich,
  \newblock Quantum hydrodynamics, {W}igner transforms, the classical limit,
  \newblock \emph{Asymptotic Analysis}, \textbf{14} (1997), 97--116.  
  
    
\bibitem{hatcher2002algebraic}
  \newblock A. Hatcher,
  \newblock \emph{Algebraic Topology},
  \newblock Cambridge University Press, 2002.  
  

\bibitem{jungel2011quasi}
  \newblock A. J{\"u}ngel,
  \newblock \emph{Quasi-Hydrodynamic Semiconductor Equations},
  \newblock Birkh{\"a}user, 2011.


\bibitem{landau1979lehrbuch}
  \newblock L. Landau and E. M. Lifschitz,
  \newblock \emph{Lehrbuch der Theoretischen Physik III - Quantenmechanik},
  \newblock Akademie-Verlag, 1979.


\bibitem{maggi2012sets}
  \newblock F. Maggi,
  \newblock \emph{Sets of Finite Perimeter and Geometric Variational Problems: An Introduction to Geometric Measure Theory},
  \newblock Cambridge University Press, 2012.
  
  
 \bibitem{markowich2010bohmian}
  \newblock P. Markowich, T. Paul and C. Sparber,
  \newblock Bohmian measures and their classical limit,
  \newblock \emph{Journal of Functional Analysis}, \textbf{259} (2010), 1542--1576. 
  

\bibitem{markowich2012dynamics}
  \newblock P. Markowich, T. Paul and C. Sparber,
  \newblock On the dynamics of {B}ohmian measures,
  \newblock \emph{Archive for Rational Mechanics and Analysis}, \textbf{205} (2012), 1031--1054.


\bibitem{massey1991basic}
  \newblock W. S. Massey,
  \newblock \emph{A Basic Course in Algebraic Topology},
  \newblock Springer Science \& Business Media, 1991.


\end{thebibliography}
\end{document}